\documentclass[reqno,a4paper,12pt]{amsart}

\usepackage[all,poly]{xy}
\usepackage{amsfonts,stmaryrd}
\usepackage[mathcal]{eucal}
\usepackage{amssymb}
\usepackage{amsmath}
\usepackage{mathrsfs}
\usepackage{color}
\usepackage[pagebackref,colorlinks]{hyperref}
\usepackage{enumerate}
\usepackage{pifont}

\theoremstyle{plain}

\newtheorem{lemma}{Lemma}
\newtheorem{The}{Theorem}
\newtheorem{Oldie}{Theorem}

\newtheorem*{corollary}{Corollary}

\newtheorem{Prob}{Problem}

\theoremstyle{remark}

\theoremstyle{definition}

\newcommand{\GL}{\operatorname{GL}}
\newcommand{\SL}{\operatorname{SL}}
\newcommand{\PGL}{\operatorname{PGL}}
\newcommand{\PE}{\operatorname{PE}}
\newcommand{\Sp}{\operatorname{Sp}}
\newcommand{\GU}{\operatorname{GU}}

\newcommand{\map}{\longrightarrow}
\newif\ifcomm
\let\ifcomm\iffalse

\def\a{\alpha}
\def\b{\beta}
\def\g{\gamma}

\def\A{\operatorname{A}}
\def\B{\operatorname{B}}
\def\C{\operatorname{C}}
\def\D{\operatorname{D}}
\def\F{\operatorname{F}}
\def\G{\operatorname{G}}
\def\E{\operatorname{E}}
\def\K{\operatorname{K}}
\def\GF#1{{\mathbb F}_{\!#1}}
\def\rk{\operatorname{rk}}

\def\pamod#1{\,(\operatorname{mod}{\, #1})\,}

\title[Commutators of relative and unrelative
elementary subgroups]{Commutators of relative and unrelative\\ 
elementary subgroups in Chevalley groups}

\author{Nikolai Vavilov}
\address{Department of Mathematics and Computer Science\\
St.~Petersburg State University\\ St.~Petersburg, Russia}
\email{nikolai-vavilov@yandex.ru}
\thanks{The work of the first author is supported by the
Russian Foundation of Fundamental Research, grant 18-31-20044.}

\author{Zuhong Zhang}
\address{Department of  Mathematics\\
 Beijing Institute of Technology\\
 Beijing, China}
\email{zuhong@hotmail.com}

\date{}

\begin{document}

\begin{abstract}
In the present paper, which is a direct sequel of our papers 
\cite{RNZ2, RNZ-generation,NZ1} joint with Roozbeh Hazrat, we achieve
a further dramatic reduction of the generating sets for commutators
of relative elementary subgroups in Chevalley groups.
Namely, let $\Phi$ be a reduced irreducible root system of rank 
$\ge 2$, let $R$ be a commutative ring and let $A,B$ be two ideals 
of $R$. We consider subgroups of the Chevalley group $G(\Phi,R)$
of type $\Phi$ over $R$. The unrelative elementary subgroup 
$E(\Phi,A)$ of level $A$ is generated (as a group) by the elementary 
unipotents $x_{\a}(a)$, $\a\in\Phi$,  $a\in A$, of level $A$. 
Its normal closure in the absolute elementary subgroup $E(\Phi,R)$ 
is denoted by $E(\Phi,R,A)$ and is called the relative elementary subgroup
of level $A$. The main results of \cite{RNZ-generation,NZ1}
consisted in construction of economic generator sets for the mutual
commutator subgroups $[E(\Phi,R,A),E(\Phi,R,B)]$, where $A$ and 
$B$ are two ideals of $R$. It turned out that one can take 
Stein---Tits---Vaserstein generators of $E(\Phi,R,AB)$, plus elementary 
commutators of the form $y_{\a}(a,b)=[x_{\a}(a),x_{-\a}(b)]$, where 
$a\in A$, $b\in B$. Here we improve these results even further, by showing 
that in fact it suffices to engage only elementary commutators 
corresponding to {\it one\/} long root, and that modulo $E(\Phi,R,AB)$
the commutators $y_{\a}(a,b)$ behave as symbols. We discuss also 
some further variations and applications of these results.
\par\smallskip\noindent

\end{abstract}

\maketitle

\maketitle
\hangindent 5.5cm\hangafter=0\noindent
To our distinguished colleague Ivan Panin,\\
a brilliant mathematician, and a wonderful friend
\par\hangindent 5.5cm\hangafter=0\noindent
\bigskip

In the present paper we continue the study of the mutual commutator
subgroups of relative subgroups in Chevalley groups. In the context of 
the general linear group $\GL(n,R)$ such commutator formulas were first 
systematically considered in the groundbreaking work by Hyman Bass
\cite{Bass2}. Soon thereafter, they were expanded to various more 
general contexts by a number of experts including Anthony Bak, 
Michael Stein, Alec Mason, Andrei Suslin, Leonid Vaserstein, Zenon
Borewicz and the first-named author, and many others. One can find
an outline of that stage in the survey \cite{RN}.
\par
The present paper continues the same general line of a long series of 
our joint papers with Roozbeh Hazrat and Alexei Stepanov, 
where we established similar {\it birelative\/} and {\it multirelative\/} 
formulas in various
contexts, see, for instance, \cite{Stepanov_Vavilov_decomposition, 
NVAS, NVAS2, RZ11, RZ12, RNZ1, RNZ3}, etc. For a much broader picture 
of the area and applications
of those resuts we refer the reader to our 
surveys and conference papers \cite{yoga, portocesareo, yoga2, RNZ4}.
\par
More specifically, the present paper is a natural sequel of our joint papers
with Hazrat, Victor Petrov and Stepanov on relative subgroups and 
commutator formulas in Chevalley groups, see
\cite{RN1, HPV, SV10,  RNZ2, RNZ-generation, Stepanov_calculus, 
Stepanov_nonabelian}, compare also the pioneering
early work by Hong You \cite{you92}. There we found, in particular,
economic generating sets for such mutual commutator subgroups
 $[E(\Phi,R,A),E(\Phi,R,B)]$, which were later used by Alexei Stepanov
 in his oustanding results on commutator width \cite{Stepanov_universal}.
\par
In 2018--2019 this line of research got an astounding boost, when we 
noticed that for $\GL(n,R)$ everything works already for the unrelativised 
groups \cite{NV18, NZ2, NZ3, NZ6}.
In \cite{NZ1} we have partially generalised these results to Chevalley 
groups, by proving that the third type of generators of $[E(\Phi,R,A),E(\Phi,R,B)]$
that occurred in \cite{RNZ-generation} are redundant. Here, we obtain yet
another dramatic improvement, and prove that it suffices to consider the 
elementary commutators $y_{\a}=[x_{\a}(a),x_{-\a}(b)]$ for a single 
{\it long\/} root and, moreover, that the classes of these elementary 
commutators modulo $E(\Phi,R,AB)$ behave as symbols in algebraic $K$-theory.


\section*{Introduction}

Let $\Phi$ be a reduced irreducible root system of rank $\ge 2$,
let $R$ be a commutative ring with 1, and let $G(\Phi,R)$ be a
Chevalley group of type $\Phi$ over $R$. For the background on
Chevalley groups over rings see \cite{NV91} or \cite{vavplot},
where one can find many further references. We fix a split maximal
torus $T(\Phi,R)$ in $G(\Phi,R)$ and consider root unipotents
$x_{\alpha}(\xi)$
elementary with respect to $T(\Phi,R)$. The subgroup $E(\Phi,R)$
generated by all $x_{\alpha}(\xi)$, where $\alpha\in\Phi$,
$\xi\in R$, is called the {\it absolute\/} elementary subgroup of
$G(\Phi,R)$.
\par
Now, let $I\unlhd R$ be an ideal of $R$. Then the 
{\it unrelativised elementary subgroup\/} $E(\Phi,I)$ of level $I$ is 
defined as the subgroup of $E(\Phi,R)$, generated by all elementary 
root unipotents $x_{\alpha}(\xi)$ of level $I$,
$$ E(\Phi,I)=
{\big\langle x_{\alpha}(\xi)\mid \alpha\in\Phi,\
\xi\in I\big\rangle}. $$
\noindent
In general, this subgroup has no chances to be normal in $E(\Phi,R)$.
Its normal closure $E(\Phi,R,I)=E(\Phi,I)^{E(\Phi,R)}$ is called the
{\it relative elementary subgroup\/} of level $I$.
\par

In the rest of this paper we impose the following umbrella assumption:

(*) In the cases $\Phi=\C_2,\G_2$ assume that $R$ does not have 
residue fields $\GF{2}$ of two elements, and in the case 
$\Phi=\C_l$, $l\ge 2$, assume additionally that any $c\in R$ 
is contained in the ideal $c^2R+2cR$.

This is precisely the condition that arises in the computation of 
the lower level of the mixed commutator subgroup $[E(\Phi,A),E(\Phi,B)]$, 
in \cite{RNZ2}, Lemma 17, 
and \cite{RNZ-generation}, Theorem 3.1, see also further related
results, and discussion of this condition in 
\cite{Stepanov_calculus, Stepanov_nonabelian}. This condition
ensures the inclusion
$$ E(\Phi,R,AB)\le [E(\Phi,A),E(\Phi,B)]. $$
\noindent
Since all vital calculations in the present paper occur modulo
$E(\Phi,R,AB)$, we are not trying to remove or weaken this condition.
In fact, when structure constants of type $\Phi$ are not invertible
in $R$, one should consider in all results more general elementary 
subgroups, corresponding to admissible pairs, rather than individual 
ideals anyway.

Let us state the main result of our previous paper \cite{NZ1},
Theorem 1.2, which, in turn, is an elaboration of the main result of
\cite{RNZ-generation}, Theorem 1.3. Below, $z_{\alpha}(a,c)$ 
are Stein---Tits---Vaserstein generators, whereas $y_{\alpha}(a,b)$
are elementary commutators, both are defined in the statement
itself, see also \S\S~1 and 2 for details.

\begin{Oldie}\label{Oldie:1}
Let $\Phi$ be a reduced irreducible root system of rank $\ge 2$.
Further, let $A$ and $B$ be two ideals of a commutative ring $R$.
Assume {\rm (*)}. 
Then the mixed commutator subgroup
$\big[E(\Phi,R,A),E(\Phi,R,B)\big]$ is generated as a group by
the elements of the form
\par\smallskip
$\bullet$ $z_{\alpha}(ab,c)=x_{-\a}(c)x_{\a}(ab)x_{-\a}(-c)$,
\par\smallskip
$\bullet$ $y_{\alpha}(a,b)=\big[x_{\alpha}(a),x_{-\alpha}(b)\big]$,
\par\smallskip\noindent
where in both cases $\alpha\in\Phi$, $a\in A$, $b\in B$, $c\in R$. 
\end{Oldie}

Recall that previous results, including \cite{RNZ-generation}, 
Theorem 1.3, required also a third type of generators for mixed
commutator subgroups, viz.\
$\big[x_{\alpha}(a),z_{\alpha}(b,c)\big]$, but the Main Lemma 
of \cite{NZ1} shows that this type of generators are redundant, 
and can be expressed as product of elementary conjugates of
the generators listed in Theorem A. Since both remaining types of 
generators sit already in $\big[E(\Phi,A),E(\Phi,B)\big]$, the above 
theorem immediately implies the following result, \cite{NZ1}, Theorem 1.1.

\begin{Oldie}\label{Oldie:2}
In conditions of Theorem~\ref{Oldie:1} 
$$ \big[E(\Phi,R,A),E(\Phi,R,B)\big]=\big[E(\Phi,A),E(\Phi,B)\big]. $$  
\end{Oldie}

Here, we obtain further striking improvements of these results.
First of all, it turns out that the set of generators in Theorem~A
can be further reduced by restricting $\a$ for the second type of
generators to a single {\it long\/} root.

\begin{The}\label{The:1}
In conditions of Theorem~\ref{Oldie:1} the mixed commutator 
of elementary 
subgroups $\big[E(\Phi,R,A),E(\Phi,R,B)\big]$ is generated as a 
group by the elements of the form
\par\smallskip
$\bullet$ $z_{\alpha}(ab,c)=x_{-\a}(c)x_{\a}(ab)x_{-\a}(-c)$,
\par\smallskip
$\bullet$ $y_{\beta}(a,b)=\big[x_{\beta}(a),x_{-\beta}(b)\big]$,
\par\smallskip\noindent
where in both cases $a\in A$, $b\in B$, $c\in R$, whereas
$\alpha\in\Phi$ is arbitrary, and $\beta\in\Phi$ is a fixed long root.
\end{The}

Morally, this theorem is also a partial counterpart of \cite{NZ1},
Theorem 4.1, which asserts that the relative elementary subgroups 
$E(\Phi,R,A)$ are themselves generated by {\it long\/} root type 
unipotents. 

Many of the auxiliary results are important and interesting in themselves, 
and we reproduce some of them in the introduction.
Firstly, it turns out that the elementary commutators are central
in $E(\Phi,R)/E(\Phi,R,AB)$.  The proof of the following result is 
similar to that of the Main Lemma in \cite{NZ1}, and in fact easier.

\begin{The}\label{The:2}
In conditions of Theorem~\ref{Oldie:1} one has
$$ {}^x y_{\a}(a,b)\equiv y_{\a}(a,b) \pamod{E(\Phi,R,AB)}. $$
\noindent
for any  $\a\in\Phi$, all $a\in A$, $b\in B$, and any $x\in E(\Phi,R)$.
\end{The}

This theorem asserts that
$$ \big[[E(\Phi,A),E(\Phi,B)],E(\Phi,R)\big]\le E(\Phi,R,AB). $$
\noindent
In particular, the quotient $[E(\Phi,A),E(\Phi,B)]/E(\Phi,R,AB)$
is itself abelian, so that we get the following result.

\begin{The}\label{The:3}
In conditions of Theorem~\ref{Oldie:1} for all
$a,a_1,a_2\in A$, $b,b_1,b_2\in B$ one has
\par\smallskip
$\bullet$ 
$y_{\a}(a_1+a_2,b)\equiv  y_{\a}(a_1,b)\cdot y_{\a}(a_2,b) 
\pamod{E(\Phi,R,AB)}$,
\par\smallskip
$\bullet$ 
$y_{\a}(a,b_1+b_2)\equiv  y_{\a}(a,b_1)\cdot y_{\a}(a,b_2) 
\pamod{E(\Phi,R,AB)}$,
\par\smallskip
$\bullet$ 
$y_{\a}(a,b)^{-1}\equiv  y_{\a}(-a,b)\equiv y_{\a}(a,-b) 
\pamod{E(\Phi,R,AB)}$,
\par\smallskip
$\bullet$ 
$y_{\a}(ab_1,b_2)\equiv y_{\a}(a_1,a_2b)\equiv e
\pamod{E(\Phi,R,AB)}$.
\end{The}

The following two results afford the advance from Theorem A
to Theorem 1. Their proofs are exactly the main novelty of
the present paper, the rest are either easy variations of the
preceding results, or easily follows.

\begin{The}\label{The:4}
In conditions of Theorem~\ref{Oldie:1} for all $a\in A$, $b\in B$, 
$c\in R$, one has\/{\rm:}
\par\smallskip
$\bullet$ $y_{\a}(a,b)\equiv y_{\b}(a,b)\pamod{E(\Phi,R,AB)}$,
\par\smallskip\noindent
for any roots $\a,\b\in\Phi$ of the same length. 
\par\smallskip
$\bullet$ $y_{\a}(a,b)\equiv{y_{\b}(a,b)}^p\pamod{E(\Phi,R,AB)}$,
\par\smallskip\noindent
if the root $\a\in\Phi$ is short, whereas the long $\b\in\Phi$ is long,
while $p=2$ for $\Phi=\B_l,\C_l,\F_4$, and $p=3$ for $\Phi=\G_2$.
\end{The}

\begin{The}\label{The:5}
In conditions of Theorem~\ref{Oldie:1} for all $a\in A$, $b\in B$, 
$c\in R$, one has\/{\rm:}
\par\smallskip
$\bullet$ $y_{\a}(ac,b)\equiv y_{\a}(a,cb)\pamod{E(\Phi,R,AB)}$,
\par\smallskip\noindent
where either $\Phi\neq\C_l$, or $\a$ is short. 
\par\smallskip
In the exceptional case when $\Phi=\C_l$ and $\a$ is long only 
the following weaker congruences hold\/{\rm:}
\par\smallskip
$\bullet$
$y_{\a}(ac^2,b)\equiv y_{\a}(a,c^2b)\pamod{E(\Phi,R,AB)}$,
\par\smallskip
$\bullet$ 
$y_{\a}(ac,b)^2\equiv y_{\a}(a,cb)^2\pamod{E(\Phi,R,AB)}$.
\end{The}

For $\GL(n,R)$ over an {\it arbitrary\/} associative ring $R$ similar 
results were established  in our recent papers \cite{NZ2, NZ3, NZ6}.
For Bak's unitary groups $\GU(2n,R,\Lambda)$, again over an 
{\it arbitrary\/} form ring $(R,\Lambda)$, such similar results are 
presently under way \cite{NZ4}.
\par
The paper is organised as follows. In \S~1 we recall necessary
notation and background.
In \S~2 we prove Theorem 2, and thus also Theorem 3. The 
technical core of the paper are \S\S~3--5, where we prove 
Theorems 4 and 5, for rank 2 root systems, $\A_2$, $\C_2$ (which
is again {\it by far\/} the most difficult case!)
and $\G_2$, respectively. Together, Theorems 2 and 4 imply 
Theorem 1. Finally, in \S~6 we derive some corollaries 
of Theorem 1 and state some further related problems.


\section{Notation and preliminary facts}\label{sec:old}

To make this paper independent of \cite{RNZ2,RNZ-generation, NZ1}, here we 
recall basic notation and the requisite facts, which will be used in 
our proofs. For more background information on Chevalley groups 
over rings, see \cite{NV91,vavplot,RN1} and references therein.

\subsection{Notation}
Let $G$ be a group. For any $x,y\in G$,  ${}^xy=xyx^{-1}$  denotes
the left $x$-conjugate of $y$. As usual, $[x,y]=xyx^{-1}y^{-1}$
denotes the [left normed] commutator of $x$ and $y$. We shall make
constant use of the obvious commutator identities, such as $[x,yz]=[x,y]\cdot{}^y[x,z]$ or $[xy,z]={}^x[y,z]\cdot[x,z]$,
usually without any specific reference.

Let $\Phi$ be a reduced irreducible root system of rank $l=\rk(\Phi)$. 
We denote by $\Phi_s$ the subset $\Phi$ consisting of short roots, 
and by $\Phi_l$ the subsystem of $\Phi$ consisting of long roots.
Fix an order on
$\Phi$ with $\Phi^{+}$, $\Phi^{-}$ and $\Pi =\{\a_{1},\ldots,\a_l\}$
being the sets of {\it positive\/}, {\it negative\/} and
{\it fundamental\/} roots, respectively. Further, let $W=W(\Phi)$
be the Weyl group of $\Phi$.

As in the introduction, we denote by $x_{\a}(\xi)$, $\a\in\Phi$, $\xi\in R$, 
the elementary generators of the (absolute) elementary Chevalley subgroup 
$E(\Phi,R)$. For a root $\a\in\Phi$ we denote by $X_{\a}$ the
corresponding [elementary] root subgroup 
$X_{\a}=\big\{x_{\a}(\xi)\mid\xi\in R\big\}$. Recall that any conjugate
${}^gx_{\a}(\xi)$ of an elementary root unipotent, where $g\in G(\Phi,R)$
is called {\it root element\/} or {\it root unipotent\/}, it is called {\it long\/} 
or {\it short\/}, depending on whether the root $\a$ itself is long or short.

Let, as in the introduction, $I$ be an ideal of $R$. We denote by
$X_{\a}(I)$ the intersection of $X_{\a}$ with the principal congruence 
subgroup $G(\Phi,R,I)$. Clearly, $X_{\a}(I)$ consists of all elementary 
root elements $x_{\a}(\xi)$, $\a\in\Phi$, $\xi\in I$, of level $I$:
$$ X_{\a}(I)=\big\{x_{\a}(\xi)\mid\xi\in I\big\}. $$
\noindent 
By definition, $E(\Phi,I)$ is generated by $X_{\a}(I)$, for all roots 
$\a\in\Phi$. The same subgroups generate $E(\Phi,R,I)$ as a {\it normal\/}
subgroup of the absolute elementary group $E(\Phi,R)$.
Generators of $E(\Phi,R,I)$ {\it as a group\/} are recalled in the next 
subsection.

\subsection{Generation of elementary subgroups}
Apart from Theorem A we shall extensively use the  two following 
generation results. The first one is a classical result by Michael Stein 
\cite{stein2}, Jacques Tits \cite{tits} and Leonid Vaserstein \cite{vaser86}. 

\begin{lemma}
Let $\Phi$ be a reduced irreducible root system of rank $\ge 2$
and let $I$ be an ideal of a commutative ring $R$. Then
as a group\/ $E(\Phi,R,I)$ is generated by the elements of
the form
$$ z_{\alpha}(a,c)=
x_{-\alpha}(c)x_{\alpha}(a)x_{-\alpha}(-c), $$
\noindent
where\/ $a\in I$, $c\in R$, and $\alpha\in\Phi$.
\end{lemma}

The following result on levels of mixed commutator subgroups
is \cite{RNZ-generation}, Theorem 4, which in turn is a sharpening
of \cite{RNZ2}, Lemmas 17--19.

\begin{lemma}\label{Lem:6}
Let\/ $\Phi$ be a reduced irreducible root system of rank\/ $\ge 2$ and let $R$ be a commutative ring.
Then for any two ideals\/ $A$ and\/ $B$
of the ring\/ $R$ one has the following inclusion
\begin{align*}
E(\Phi,R,AB)\le\big[E(\Phi,A),E(\Phi,B\big]&\le\big [E(\Phi,R,A),E(\Phi,R,B)\big]\\&\le
\big[G(\Phi,R,A),G(\Phi,R,B)\big]\le G(\Phi,R,AB).
\end{align*}
\end{lemma}

\subsection{Structure constants}
All results of the present paper are based on the Steinberg 
relations among the elementary generators, which will be repeatedly 
used without any specific reference. Especially important for us is
the Chevalley commutator formula 
$$ [x_{\a}(a),x_{\beta}(b)]=\prod_{i\a+j\beta\in \Phi}
x_{i\a+j\beta}(N_{\a\beta ij}a^ib^j), $$
\noindent
where $\a\not=-\beta$ and $N_{\a\beta ij}$ are the structure constants
which do not depend on $a$ and $b$. However, for $\Phi=\G_2$
they may depend on the order of the roots in the product on the right
hand side. See \cite{carter,stein2,steinberg,vavplot} for more
details regarding the structure constants $N_{\a\beta ij}$.
\par
In the proof of Theorems 4 and 5 we need somewhat more specific 
information about the structure constants. For $\Phi=\A_2$ and
$\Phi=\C_2$ this is easy, since the corresponding simply connected
Chevalley groups can be identified with $\SL(3,R)$ and $\Sp(4,R)$,
respectively, and we select the usual parametrisation of the elementary
root subgroups therein. 
\par
For $\Phi=\C_2$ the most complicated instance of the Chevalley 
commutator formula 
is when  $\a$ and $\b$ are the long and short fundamental roots, 
respectively. We will choose the parametrisation of root subgroups
for which
$$ [x_{\a}(a),x_{\b}(b)]=x_{\a+\b}(ab)x_{\a+2\b}(ab^2). $$
\par
The case of $\Phi=\G_2$ is somewhat more tricky. Let $\a$ and $\b$ 
be the short and long fundamental roots, respectively. Then it is known 
that the parametrisation of the root subgroups can be chosen in such 
a way that
\begin{align*}
&[x_{\a}(a),x_{\b}(b)]=
x_{\a+\b}(ab)x_{2\a+\b}(a^2b)x_{3\a+\b}(a^3b)x_{3\a+2\b}(2a^3b^2)\\
&[x_{\a}(a),x_{\a+\b}(b)]=
x_{2\a+\b}(2ab)x_{3\a+\b}(3a^2b)x_{3\a+2\b}(-3ab^2),\\
&[x_{\a}(a),x_{2\a+\b}(b)] = x_{3\a+\b}(3ab),\\
&[x_{\b}(a),x_{3\a+\b}(b)] = x_{3\a+2\b}(ab),\\
&[x_{\a+\b}(a),x_{2\a+\b}(b)] = x_{3\a+2\b}(-3ab),
\end{align*}
\noindent
these are precisely the signs you get for the {\it positive\/} Chevalley base.
See, for instance \cite{carter, steinberg, vavplot}.
\par
Our initial proof of Theorems 4 and 5 in the case $\Phi=\G_2$ relied
on an explicit knowledge of the structure constants also in some further
instances of the Chevalley commutator formula. Initially, we used a 
{\tt Mathematica} package
{\tt g2.nb} by Alexander Luzgarev, to compute the structure constants.
However, later we noticed that pairs of short roots do not require a
separate analysis. This is precisely the shortcut presented in \S~5
below.

\subsection{Parabolic subgroups} As in \cite{NZ3}
an important part in the proof of Theorems 2, 4 and 5
is played by the Levi decomposition for [elementary] parabolic 
subgroups. Oftentimes, it allows us to discard factors in the unipotent 
radicals, to limit the number of instances, 
where we have to explicitly invoke precise forms of the 
Chevalley commutator formula.

Classical Levi decomposition asserts that any parabolic subgroup $P$ of
$G(\Phi,R)$ can be expressed as the semi-direct product
$P=L_P\rightthreetimes U_P$ of its unipotent radical
$U_P \unlhd P$ and a Levi subgroup $L_P\le P$. However, as in
\cite{RNZ-generation, NZ1} we do not have to recall the general case. 
\par\smallskip
$\bullet$ Since we calculate inside $E(\Phi,R)$, we can limit
ourselves to the {\it elementary\/} parabolic subgroups,
spanned by some root subgroups $X_{\a}$.
\par\smallskip
$\bullet$ Since we can choose the order on $\Phi$ arbitrarily,
we can always assume that $\a$ is fundamental and, thus, limit 
ourselves to {\it standard\/} parabolic subgroups.
\par\smallskip
$\bullet$ Since the proofs of our main results reduces to groups
of rank 2, we could only consider rank 1 parabolic subgroups,
which {\it in this case\/} are maximal parabolic subgroups.
\par\smallskip
Thus, we consider only {\it elementary\/} rank 1 parabolics, which 
are defined as follows. Namely, we fix an order on $\Phi$, and
let $\Phi^+$ and $\Phi^-$ be the corresponding sets of positive 
and negative roots, respectively. Further, let $\Pi=\{\a_1,\ldots,\a_l\}$
be the corresponding fundamental system. For any $r$, $1\le r\le l$, 
and define the $r$-th rank 1 {\it elementary\/} parabolic subgroup as
$$ P_{\a_r}=\langle U,X_{-\a_r}\rangle\le E(\Phi,R). $$
\noindent
Here $U= \prod X_{\a}$, $\a\in\Phi^+$, is the unipotent radical of
the standard Borel subgroup $B$. Then the unipotent radical of $P_{\a_r}$
has the form
$$ U_{\a_r}=\prod X_{\a},\quad \a\in\Phi^+,\  \a\neq\a_r, $$
\noindent
whereas $L_{\a_r}=\langle X_{\a_r},X_{-\a_r}\rangle$ is the [standard] 
Levi subgroup of $P_r$. Clearly, $L_{\a_r}$ is isomorphic to the elementary 
subgroup $E(2,R)$ in $\SL(2,R)$, or to its projectivised version 
$\PE(2,R)$ in $\PGL(2,R)$. In the sequel we usually (but not always!)
abbreviate $P_{\a_r},U_{\a_r},L_{\a_r}$, etc., to $P_r,U_r,L_r$, etc.
\par
Levi decomposition (which in the case of elementary parabolics 
immediately follows from
the Chevalley commutator formula) asserts that the group $P_r$ 
is the semi-direct product $P_r=L_r\rightthreetimes U_r$ of 
$U_r\unlhd P_r$ and $L_r\le P_r$. The most important part is the 
[obvious] claim is that $U_r$ is normal in $P_r$.
\par
Simultaneously with $P_r$ one considers also the opposite parabolic 
subgroup $P_r^-$ defined as
$$ P_r^-=\langle U^-,X_{\a_r}\rangle\le E(\Phi,R). $$
\noindent
Here $U^-= \prod X_{\a}$, $\a\in\Phi^-$, is the unipotent radical of
the Borel subgroup $B^-$ opposite to the standard one. Clearly,
$P_r$ and $P_r^-$ share the common [standard] Levi subgroup $L_r$, 
whereas the unipotent radical $U_r^-$ of $P_r^-$ is opposite to that 
of $P_r$, and has the form
$$ U_r^-=\prod X_{\a},\quad \a\in\Phi^-,\  \a\neq-\a_r. $$
\noindent
Now, Levi decomposition takes the form
$P_r^-=L_r\rightthreetimes U_r^-$ with $U_r^-\unlhd P_r^-$. In other
words, $U_r$ and $U_r^-$  are both normalised by $L_r$. 

Actually, we need a slightly more precise form of this last statement.
Namely, let $I$ be an ideal of $R$. Denote by $L_r(I)$ the principal
congruence subgroup of level $I$ in $L_r$ and by $U_r(I)$ and 
$U_r^-(I)$ the respective intersections of $U_r$ and $U_r^-$ with
$G(\Phi,R,I)$ --- or, what is the same, with $E(\Phi,R,I)$:
$$ U_r(I)=U_r\cap E(\Phi,R,I),\qquad U_r^-(I)=U_r^-\cap E(\Phi,R,I). $$
\noindent
Obviously, $U_r(I),U_r^-(I)\le E(\Phi,I)$ are normalised by $L_r$. 

The following fact is well known, and obvious. 

\begin{lemma}
Let $A$ and $B$ be two ideals of $R$. Then 
$$ [L_r(A),U_r(B)]\le U_r(AB),\qquad [L_r(A),U_r^-(B)]\le U_r^-(AB). $$
\noindent
In particular, both commutators are contained in 
$E(\Phi,AB)\le E(\Phi,R,AB)$.
\end{lemma}


\section{Proof of Theorems 2 and 3}\label{sec:proof2}


This section is devoted to the proof of Theorem 2. It is a 
calculation of the same type as the proof of the Main Lemma
in \cite{RN1}, and actually easier than that, since now we can 
expand the exponent level, rather than the ground level, so
that there are no protruding factors that have to be taken care of
and the elementary commutators do not procreate.

\subsection{Idea of the proof}
Consider the elementary conjugate ${}^xy_{\a}(a,b)$. We argue by induction on the length of $x\in E(\Phi,R)$ in elementary generators. 
Let $x=yx_{\b}(c)$, where $y\in E(\Phi,R)$ is shorter than $x$, 
whereas $\b\in\Phi$, $c\in R$. First of all, recall that under the action 
of the Weyl group $W(\Phi)$ the root $\a$ is conjugate to a fundamental
root of the same length. Thus, we could from the very start choose an ordering of $\Phi$ such that $\a=\a_r\in\Pi$ is a fundamental root for
some $r$, $1\le r\le l$.
\par\smallskip

If $\b\neq\pm\a$, then $x_{\b}(c)$ belongs either 
to $U_r$, or to $U_r^-$. By Lemma 3 in each case 
$[x_{\b}(c),y_{\a}(a,b)]\in E(\Phi,R,AB)$ and thus
$$ {}^{x_{\b}(c)}y_{\a}(a,b)=
[x_{\b}(c),y_{\a}(a,b)]\cdot y_{\a}(a,b)
\equiv y_{\a}(a,b)\pamod{E(\Phi,R,AB)}. $$
\par\smallskip

\subsection{Expansion of the exponent}
It remains only to consider the case, where 
$\b=\pm\a$. In each case we will express $x_{\b}(c)$ as
a product of root elements satisfying the conditions of the
previous item. One of the four following possibilities may 
occure. Since we are only looking at {\it one\/} instance
of the Chevalley commutator formula at a time, the 
parametrisation of the corresponding root subgroups can 
be chosen in such a way that all the resulting structure 
constants are positive (see \cite{carter, steinberg} or
 \cite{vavplot} and references there.
\par\smallskip
$\bullet$ First, assume that $\a$ can be embedded into 
a subsystem of type $\A_2$. This already proves Theorem~1 
for simply laced Chevalley groups, and for the Chevalley
group of type $\F_4$. It also proves necessary congruences 
for a {\it short\/} root $\a$ in Chevalley groups of type $\C_l$, 
$l\ge 3$, for a {\it long\/} root $\a$ in Chevalley groups of 
type $\B_l$, $l\ge 3$, and for a {\it long\/} root $\a$ in the 
Chevalley group of type $\G_2$.
\par
In this case there exist roots $\g,\delta\in\Phi$, of the 
same length as $\a$ such that $\b=\g+\delta$ and $N_{\g\delta11}=1$. Express $x_{\b}(c)$ in the form 
$x_{\b}(c)=[x_{\g}(1),x_{\delta}(c)]$
and plug this expression in the exponent. We get
$$ {}^{x_{\b}(c)}y_{\a}(a,b)=
{}^{x_{\g}(1)x_{\delta}(c)x_{\g}(-1)x_{\delta}(-c)}y_{\a}(a,b)
\equiv y_{\a}(a,b)\pamod{E(\Phi,R,AB)},  $$
\noindent
by the first item in the proof.
\par\smallskip
$\bullet$ Next, assume that $\a$ can be embedded into a
subsystem of type $\C_2$ as a {\it short\/} root. In this case
we express $\b$ as $\b=\g+\delta$, where $\g$ is long and 
$\delta$ is short. By the above we may $x_{\b}(c)$ in the form 
$$ x_{\b}(c)=[x_{\g}(c),x_{\delta}(1)]\cdot x_{\g+2\delta}(-c). $$
\noindent
Plugging this expression in the exponent, we get 
$$ {}^{x_{\b}(c)}y_{\a}(a,b)=
{}^{x_{\g}(1)x_{\delta}(c)x_{\g}(-1)x_{\delta}(-c)
x_{\g+2\delta}(-c)}y_{\a}(a,b)
\equiv y_{\a}(a,b)\pamod{E(\Phi,R,AB)},  $$
\noindent
where again $\g,\delta,\g+2\delta\neq\pm\a$, so that we
can invoke the first item.
\par\smallskip
$\bullet$  Next, assume that $\a$ can be embedded into a
subsystem of type $\C_2$ as a {\it long\/} root. In this case
we express $\b$ as $\b=\g+2\delta$, with the same $\g,\delta$ 
as above. so that the formula takes the form
$$ x_{\b}(c)=[x_{\g}(c),x_{\delta}(1)]\cdot x_{\g+\delta}(-c). $$
\noindent
Plugging this expression in the exponent, we get 
$$ {}^{x_{\b}(c)}y_{\a}(a,b)=
{}^{x_{\g}(1)x_{\delta}(c)x_{\g}(-1)x_{\delta}(-c)
x_{\g+\delta}(-c)}y_{\a}(a,b)
\equiv y_{\a}(a,b)\pamod{E(\Phi,R,AB)},  $$
\noindent
where again $\g,\delta,\g+\delta\neq\pm\a$.
\par\smallskip
$\bullet$ Finally, when $\a$ is a {\it short\/} root in $\G_2$,
$\b$ as $\b=\g+\delta$, where $\g$ is long and 
$\delta$ is short. By the above, we can rewrite the Chevalley
commutator formula in the form
$$ x_{\b}(c)=[x_{\g}(c),x_{\delta}(1)]\cdot
x_{\g+2\delta}(-c)x_{\g+3\delta}(-c)x_{2\g+3\delta}(-2c^2). $$
\noindent
Plugging this expression in the exponent, we get 
\begin{multline*}
 {}^{x_{\b}(c)}y_{\a}(a,b)=
{}^{x_{\g}(1)x_{\delta}(c)x_{\g}(-1)x_{\delta}(-c)
x_{\g+2\delta}(-c)x_{\g+3\delta}(-c)x_{2\g+3\delta}(-2c^2)}
y_{\a}(a,b)\equiv\\
 y_{\a}(a,b)\pamod{E(\Phi,R,AB)},  
\end{multline*}
\noindent
where again $\g,\delta,\g+2\delta,\g+3\delta,2\g+3\delta\neq\pm\a$.

\subsection{Proof of Theorems 2 and 3}
Summarising the above, we see that for all elementary generators $x_{\b}(c)$ one has 
${}^{x_{\b}(c)}y_{\a}(a,b)\equiv
 y_{\a}(a,b)\pamod{E(\Phi,R,AB)}$ and thus
$$ {}^xy_{\a}(a,b)\equiv {}^yy_{\a}(a,b) \pamod{E(\Phi,R,A\B)}, $$
\noindent
where the length of $y$ in elementary generators is smaller than 
the length of $x$.
By induction we get that ${}^xy_{\a}(a,b)\equiv y_{\a}(a,b)\pamod{E(\Phi,R,A\B)}$, as claimed. This proves Theorem 2. 
\par
It is clear that Theorem 3 immediately follows. Indeed, to derive the first 
item, observe that
$$ y_{\a}(a_1+a_2,b)=[x_{\a}(a_1+a_2),x_{-\a}(b)]= 
[x_{\a}(a_1)x_{\a}(a_2),x_{-\a}(b)]. $$
\noindent
Using multiplicativity of the commutator w.~r.`t.\ the first argument, 
we see that
$$ y_{\a}(a_1+a_2,b)={}^{x_{\a}(a_1)}[x_{\a}(a_2),x_{-\a}(b)]\cdot
[x_{\a}(a_1),x_{-\a}(b)]=
{}^{x_{\a}(a_1)}y_{\a}(a_2,b)\cdot y_{\a}(a_1,b). $$
\noindent
It remains to apply Theorem 2.  The second item is similar, and the 
third item follows. The last item is obvious from the definition.


\section{Proof of Theorems 4 and 5: the case $\A_2$}

We are now all set to take up the proof of Theorems 4 and 5. 
In the present section we prove Theorems 4 and 5 for simply laced
systems.

\subsection{Structure of the proof} The proof will be 
subdivided into a sequence of five lemmas, which either simultaneously
establish congruences in Theorems 4 and 5, for some pairs of roots of
the same length, or reduce elementary commutators for short roots to 
elementary commutators for long roots. These five cases are:
\par\smallskip
$\bullet$ Two roots $\a$ and $\b$ that can be embedded into $\A_2$, Lemma 4,
\par\smallskip
$\bullet$ Two short roots in $\C_2$, Lemma 5,
\par\smallskip
$\bullet$ Two long roots in $\C_2$, Lemma 6,
\par\smallskip
$\bullet$ A short and a long root in $\C_2$, Lemma 7,
\par\smallskip
$\bullet$ A short and a long root in $\G_2$, Lemma 8.
\par\smallskip\noindent
Already Lemma 4 suffices to establish Theorem 4, and thus
also Theorem 1, for the case of simply-laced systems. It
also reduces both long and short elementary commutators 
in $\F_4$, long elementary commutators in $\B_l$, $l\ge 3$, 
and $\G_2$ and short elementary commutators in 
$\C_l$, $l\ge 3$, to such elementary commutators for a single 
root of that length. After that, Lemmas 5--7 
completely settle the case of doubly laced root systems. Finally, 
Lemma 8 is only needed for $\G_2$. Observe that together with 
Lemma 4 it immediately implies also the necessary congruences
for pairs of  short roots in $\G_2$. 
\par\smallskip\noindent
{\bf Warning.} A similar strategy does not work for $\C_2$ since
in this case the congruences for long roots in Theorem 5 are 
{\it weaker\/}, than the desired congruences for short roots. 
This compels us to derive the congruences for pairs of short 
roots and for pairs of long roots independently, {\it before\/} comparing
elementary commutators for short roots with those for long roots. This
makes $\C_2$ the most exacting case of all.

\subsection{Two roots in $\A_2$}
The first of these lemmas was essentially contained already in
\cite{NZ2}, Lemma 5, and \cite{NZ3}, Lemma 11. Of course,
there we used matrix language. For the sake of completeness, 
and also as a template for the following more difficult lemmas,
below we reproduce its proof in the language of roots.


\begin{lemma}
Assume that the roots $\a,\b\in\Phi$ of the same length
can be embedded into a subsystem of type $\A_2$. Then
for all $a\in A$, $b\in B$, $c\in R$, one has\/{\rm:}
$$ y_{\a}(ac,b)\equiv y_{\b}(a,cb)\pamod{E(\Phi,R,AB)}. $$
\end{lemma}
\begin{proof}
First, assume that $\b$ is such that $\a=\b+\g$, with
$N_{\b\g11}=1$ and rewrite the elementary commutator
$ y_{\a}(ac,b)=\big[x_{\a}(ac),x_{-\a}(b)\big]$ as
$$  y_{\a}(ac,b)=x_{\a}(ac)\cdot{}^{x_{-\a}(b)}x_{\a}(-ac)=
x_{\a}(ac)\cdot{}^{x_{-\a}(b)}\big[x_{\b}(a),x_{\g}(-c)\big]. $$
\noindent
Expanding the conjugation by $x_{-\a}(b)$, we see that 
$$  y_{\a}(ac,b)=x_{\a}(ac)\cdot\big[{}^{x_{-\a}(b)}x_{\b}(a),{}^{x_{-\a}(b)}x_{\g}(-c)\big] = 
x_{\a}(ac)\cdot\big[x_{-\g}(ba)x_{\b}(a),x_{\g}(-c)x_{-\b}(cb)\big]. $$
\noindent
Now, the first factor $x_{-\g}(ba)$ of the first argument in this last commutator already belongs to the group $E(\Phi,AB)$ which is 
contained in $E(\Phi,R,AB)$. Thus, as above, 
$$  y_{\a}(ac,b)\equiv  x_{\a}(ac)\cdot\big[x_{\b}(a),x_{\g}(-c)x_{-\b}(cb)\big] \pamod{E(\Phi,R,AB)}. $$
\noindent
Using multiplicativity of the commutator w.~r.~t.\ the second argument, 
cancelling the first two factors of the resulting 
expression, and then applying Theorem~2 we see that
for a pair of roots $\a,\b$ at angle $\pi/3$, one has
$$  y_{\a}(ac,b)\equiv 
{}^{x_{\g}(-c)}\big[x_{\b}(a),x_{-\b}(cb)\big] 
\equiv \big[x_{\b}(a),x_{-\b}(cb)\big]\equiv y_{\b}(a,cb)
\pamod{E(\Phi,R,AB)}, $$
\noindent
as claimed.
Obviously, one can pass from any root in $\A_2$ to any other
such root in not more than 3 such elementary steps.
\end{proof}

Joining two roots of the same length by a sequence of roots
where every two neighbours sit in a subsystem of type $\A_2$, 
we obtain the following corollary.

\begin{corollary}
Assume that the roots $\a,\b\in\Phi$ of the same length
and one of the following holds\/{\rm:}
\par\smallskip
$\bullet$ $\Phi=\A_l,\D_l,\E_l,\F_4$, 
\par\smallskip
$\bullet$ $\Phi=\B_l$, $l\ge 3$, and $\a,\b$ are long,
\par\smallskip
$\bullet$ $\Phi=\C_l$, $l\ge 3$, and $\a,\b$ are short.
\par\smallskip
$\bullet$ $\Phi=\G_2$, and $\a,\b$ are long,
\par\smallskip\noindent
Then for all $a\in A$, $b\in B$, $c\in R$, one has\/{\rm:}
$$ y_{\a}(ac,b)\equiv y_{\b}(a,cb)\pamod{E(\Phi,R,AB)}. $$
\end{corollary}

The remaining cases have to be considered separately,
in the same style, as Lemma 4. However, in these cases
the roots $\b$ and $\g$ in the proof of this lemma would have
different lengths, so that it does matter, whether we put
parameter $a$ in the above calculation in the short root unipotent, 
or the long root unipotent. In fact, by choosing one way, or the other,
one gets different congruences! Also, in the case $\Phi=\G_2$ the 
structure constants have to be chosen in consistent way.


\section{Proof of Theorems 4 and 5: the case $\C_2$}

In this section we prove Theorems 4 and 5 for doubly laced 
systems. This is by far the most difficult case of all, since in 
this case we have to consider short roots and long roots
separately.

\subsection{Two short roots}
The following lemma settles the case of short roots in $\B_l$, 
$l\ge 2$.

\begin{lemma}
Assume that the roots $\a,\b\in\Phi$ can be embedded as
short roots into a subsystem of type $\C_2$. Then
for all $a\in A$, $b\in B$, $c\in R$, one has\/{\rm:}
$$ y_{\a}(ac,b)\equiv y_{\b}(a,cb)\pamod{E(\Phi,R,AB)}. $$
\end{lemma}
\begin{proof}
First assume that $\a$ and $\b$ are linearly independent.
Then there exists a long root $\g$ such that $\a=\b+\g$ and
we can choose parametrisation of root subgroups such that 
$N_{\b\g11}=N_{\b\g21}=1$. Actually, the signs mostly do 
not play any role here, apart from one position. Namely, we
should eventually get that $y_{\a}(ac,b)$ is equivalent to 
$y_{\b}(a,cb)$, and not to ${y_{\b}(a,cb)}^{-1}$. They 
were calculated in $\Sp(4,R)$.
\par
Expanding the elementary 
commutator $ y_{\a}(ac,b)$ as in Lemma 4 and plugging in 
$x_{\a}(-ac)=x_{\a+\b}(a^2c)[x_{\b}(a),x_{\g}(-c)]$, 
we get
$$  y_{\a}(ac,b)=x_{\a}(ac)\cdot{}^{x_{-\a}(b)}x_{\a}(-ac)=
x_{\a}(ac)\cdot {}^{x_{-\a}(b)}x_{\a+\b}(a^2c)
\cdot {}^{x_{-\a}(b)}\big[x_{\b}(a),x_{\g}(-c)\big]. $$
\noindent
Expanding the conjugation by $x_{-\a}(b)$, we see that 
$$  y_{\a}(ac,b)= 
x_{\a}(ac)\cdot {}^{x_{-\a}(b)}x_{\a+\b}(a^2c)\cdot
\big[x_{-\g}(ba)x_{\b}(a),x_{\g}(-c)x_{-\b}(cb)x_{-\a-\b}(cb^2)\big]. $$
\par
Now, the first factor $x_{-\g}(ba)$ of the first argument in this last commutator already belongs to the group $E(\Phi,AB)$ which is 
contained in $E(\Phi,R,AB)$. Also, 
$$ {}^{x_{-\a}(b)}x_{\a+\b}(a^2c)\equiv x_{\a+\b}(a^2c)\pamod{E(\Phi,R,AB)}. $$
\noindent
Thus, as above, 
$$  y_{\a}(ac,b)\equiv  x_{\a}(ac)x_{\a+\b}(a^2c)\cdot[x_{\b}(a),x_{\g}(-c)x_{-\b}(cb)x_{-\a-\b}(cb^2)\big] \pamod{E(\Phi,R,AB)}. $$
\noindent
Using multiplicativity of the commutator w.~r.~t.\ the second argument, 
cancelling the first commutator of the resulting 
expression, we see that
\begin{multline*}
y_{\a}(ac,b)\equiv 
\big[x_{\b}(a),x_{-\b}(cb)x_{-\a-\b}(cb^2)\big]=\\y_{\b}(a,cb)\cdot
{}^{x_{-\b}(cb)}\big[x_{\b}(a),x_{-\a-\b}(cb^2)\big]
\equiv y_{\b}(a,cb)
\pamod{E(\Phi,R,AB)}. 
\end{multline*}
\par
Obviously, one can pass from a short root $\a$ in $\C_2$ to the opposite 
root $-\a$ in two such elementary steps.
\end{proof}


\subsection{Two long roots}
The following lemma settles the case of long roots in $\C_l$, 
$l\ge 2$. This case is exceptional, since here, unlike all other
cases, the arguments of an elementary commutator are only
balanced up to squares. In the following lemma we establish 
the first related congruence in Theorem 5.

\begin{lemma}
Assume that the roots $\a,\g\in\Phi$ can be embedded as
long roots into a subsystem of type $\C_2$. 
Then for all $a\in A$, $b\in B$, $c\in R$, one has\/{\rm:}
$$ y_{\a}(ac^2,b)\equiv y_{\g}(a,c^2b)\pamod{E(\Phi,R,AB)}. $$
\end{lemma}
\begin{proof}
First, let $\a$ and $\g$ be linearly independent long roots. As in
the previous lemma we choose a short root $\b$ such that 
$\a=2\b+\g$ and specify the same choice of signs. 
\par
Expanding the elementary 
commutator $ y_{\a}(ac^2,b)$ as in Lemma 4 and plugging in 
$x_{\a}(-ac^2)=x_{\g+\b}(ac)[x_{\b}(c),x_{\g}(-a)]$, 
we get
$$  y_{\a}(ac^2,b)=x_{\a}(ac^2)\cdot{}^{x_{-\a}(b)}x_{\a}(-ac^2)=
x_{\a}(ac^2)\cdot {}^{x_{-\a}(b)}x_{\g+\b}(ac)
\cdot {}^{x_{-\a}(b)}\big[x_{\b}(c),x_{\g}(-a)\big]. $$
\noindent
Expanding the conjugation by $x_{-\a}(b)$, we see that 
$$  y_{\a}(ac^2,b)= 
x_{\a}(ac^2)\cdot {}^{x_{-\a}(b)}x_{\g+\b}(ac)\cdot
\big[x_{\b}(c)x_{-\b-\g}(cb)x_{-\g}(c^2b),x_{\g}(-a)\big]. $$
\par
As usual,
$$ {}^{x_{-\a}(b)}x_{\g+\b}(ac)\equiv x_{\g+\b}(ac)\pamod{E(\Phi,R,AB)}. $$
\noindent
so that the first two factors of the above expression are the 
inverse of $[x_{\b}(c),x_{\g}(-a)]$. Thus, up to a congruence
modulo $E(\Phi,R,AB)$ we get 
\begin{multline*}
y_{\a}(ac^2,b)\equiv 
\big[x_{-\b-\g}(cb)x_{-\g}(c^2b),x_{\g}(-a)\big] \equiv\\
y_{-\g}(c^2b,-a)\equiv y_{\g}(a,c^2b)
\pamod{E(\Phi,R,AB)}.
\end{multline*}
\par
Obviously, one can pass from a long root $\a$ in $\C_2$ to the opposite 
root $-\a$ in two such elementary steps.
\end{proof}


\subsection{A short root and a long root}
The following lemma establishes connection between the classes 
of short and long elementary commutators in doubly laced 
systems.

\begin{lemma}
Assume that the roots $\a,\g\in\Phi$ can be embedded as
a short root and a long root into a subsystem of type $\C_2$. 
Then for all $a\in A$, $b\in B$, $c\in R$, one has\/{\rm:}
$$ y_{\a}(ac,b)\equiv y_{\g}(a,cb)^2\pamod{E(\Phi,R,AB)}. $$
\end{lemma}
\begin{proof}
First, assume that $\a$ and $\g$ form an angle $\pi/4$. 
We choose a short root $\b$ such that $\a=\b+\g$ and 
specify the same choice of signs. 
\par
Expanding the elementary 
commutator $ y_{\a}(ac,b)$ as in Lemma 5 and plugging in 
$x_{\a}(-ac)=x_{\a+\b}(-ac^2)[x_{\b}(-c),x_{\g}(a)]$, 
we get
$$  y_{\a}(ac,b)=x_{\a}(ac)\cdot{}^{x_{-\a}(b)}x_{\a}(-ac)=
x_{\a}(ac)\cdot {}^{x_{-\a}(b)}x_{\a+\b}(-ac^2)
\cdot {}^{x_{-\a}(b)}\big[x_{\b}(-c),x_{\g}(a)\big]. $$
\noindent
Expanding the conjugation by $x_{-\a}(b)$, we see that 
$$  y_{\a}(ac,b)= 
x_{\a}(ac)\cdot {}^{x_{-\a}(b)}x_{\a+\b}(-ac^2)\cdot
\big[x_{-\g}(-2cb)x_{\b}(-c),x_{\g}(a)x_{-\b}(-ab)x_{-\a-\b}(-ab^2)\big]. $$
\par
Now, the last two factors $x_{-\b}(-ab)x_{-\a-\b}(-ab^2)$ of the 
second argument in this last commutator already belong to the 
group $E(\Phi,AB)$ which is contained in $E(\Phi,R,AB)$. Also, 
$$ {}^{x_{-\a}(b)}x_{\a+\b}(-ac^2)\equiv x_{\a+\b}(-ac^2)\pamod{E(\Phi,R,AB)}. $$
\noindent
Thus, as above, 
$$  y_{\a}(ac,b)\equiv  x_{\a}(ac)x_{\a+\b}(-ac^2)
\cdot\big[x_{-\g}(-2cb)x_{\b}(-c),x_{\g}(a)\big] \pamod{E(\Phi,R,AB)}. $$
\noindent
Using multiplicativity of the commutator w.~r.~t.\ the first argument, and
cancelling the first commutator of the resulting expression, we see that
$$
y_{\a}(ac,b)\equiv y_{-\g}(-2cb,a)\equiv y_{-\g}(cb,a)^{-2}
\equiv y_{\g}(a,cb)^2
\pamod{E(\Phi,R,AB)}. $$
\par
Obviously, combined with the previous lemma this gives necessary 
inlcusions for all pairs of a short and a long root.
\end{proof}

\begin{corollary}
Assume that the roots $\a,\g\in\Phi$ can be embedded as
long roots into a subsystem of type $\C_2$. Then
for all $a\in A$, $b\in B$, $c\in R$, one has\/{\rm:}
$$ y_{\a}(ac,b)^2\equiv y_{\b}(a,cb)^2\pamod{E(\Phi,R,AB)}. $$
\end{corollary}
\begin{proof} 
Indeed, let $\g$ be any short root. Then by the previous lemma
and Lemma 5 one has
$$ y_{\a}(ac,b)^2\equiv y_{\g}(ac,b)\equiv
y_{\g}(a,cb)\equiv y_{\b}(a,cb)^2\pamod{E(\Phi,R,AB)}. $$
\end{proof}
This completes the proof of Theorems 4 and 5 for doubly laced 
root systems. 


\section{Proof of Theorems 4 and 5: the case $\G_2$}

In this section we finish the proof of Theorems 4 and 5 for the
only remaining case $\Phi=\G_2$. Since in this case long roots 
themselves form a root system of type $\A_2$, the corresponding
elementary commutators are balanced with respect to all 
elements of $R$, which makes the proof quite a bit easier.

The following lemma establishes connection between the classes 
of short and long elementary commutators in $\G_2$.

\begin{lemma}
Assume that $\a,\g\in\G_2$, where $\a$ is short and $\g$ is
long. Then for all $a\in A$, $b\in B$, $c\in R$, one has\/{\rm:}
$$ y_{\a}(ac,b)\equiv y_{\g}(a,cb)^3\pamod{E(\Phi,R,AB)}. $$
\end{lemma}
\begin{proof}
First, assume that $\a$ and $\g$ form an angle $\pi/6$. 
We choose a short root $\b$ such that $\a=\b+\g$ and 
specify the same choice of signs as in Section 1.3.
\par
Expanding the elementary 
commutator $ y_{\a}(ac,b)$ as in Lemma 4 and plugging in 
$x_{\a}(-ac)=u\cdot [x_{\b}(-c),x_{\g}(a)]$, where 
$u=x_{\a+\b}(-ac^2)x_{\a+2\b}(ac^3)x_{2\a+\b}(2a^2c^3)$,
we get
$$  y_{\a}(ac,b)=x_{\a}(ac)\cdot{}^{x_{-\a}(b)}x_{\a}(-ac)=
x_{\a}(ac)\cdot {}^{x_{-\a}(b)}u
\cdot {}^{x_{-\a}(b)}\big[x_{\b}(-c),x_{\g}(a)\big]. $$
\noindent
Clearly, ${}^{x_{-\a}(b)}u\equiv u\pamod{E(\Phi,R,AB)}$.

Expanding the conjugation by $x_{-\a}(b)$, we see that 
$y_{\a}(ac,b)= x_{\a}(ac)\cdot {}^{x_{-\a}(b)}u\cdot v$,
where
$$ v=
\big[x_{-\g}(-3cb)x_{\b}(-c),x_{\g}(a)x_{-\a-2\b}(-a^2b^3)x_{-2\a-\b}(ab^3)x_{-\a-\b}(ab^2)
x_{-\b}(ab)
\big]. $$
\noindent
Clearly, the last four factors of the second argument in this last 
commutator already be\-long to the group $E(\Phi,AB)$ which 
is contained in $E(\Phi,R,AB)$. 

\par
Thus, by the same token, as above,
$$  y_{\a}(ac,b)\equiv  x_{\a}(ac)\cdot u
\cdot\big[x_{-\g}(-3cb)x_{\b}(-c),x_{\g}(a)\big] \pamod{E(\Phi,R,AB)}. $$
\noindent
Using multiplicativity of the commutator w.~r.~t.\ first argument, cancelling 
the first commutator of the resulting expression, we see that
$$
y_{\a}(ac,b)\equiv y_{-\g}(-3cb,a)\equiv y_{-\g}(cb,a)^{-3}
\equiv y_{\g}(a,cb)^3
\pamod{E(\Phi,R,AB)}. $$
\par
Obviously, combined with Lemma 4 this gives necessary 
inlcusions for all pairs of a short and a long root.
\end{proof}

\begin{corollary}
Assume that the roots $\a,\b\in\G_2$. Then
for all $a\in A$, $b\in B$, $c\in R$, one has\/{\rm:}
$$ y_{\a}(ac,b)\equiv y_{\b}(a,cb)\pamod{E(\Phi,R,AB)}. $$
\end{corollary}
\begin{proof} 
Indeed, let $\g$ be any long root. Then by the previous lemma
and Lemma 4 one has
$$ y_{\a}(ac,b)\equiv y_{\g}(ac,b)^3\equiv
y_{\g}(a,cb)^3\equiv y_{\b}(a,cb)\pamod{E(\Phi,R,AB)}. $$
\end{proof}
This completes the proof of Theorems 4 and 5 for the only remaining
case $\Phi=\G_2$, and thus also the proof of Theorem 1, for
all cases.


\section{Final remarks}\label{sec:final}

Theorem 1 implies surjective stability for the abelian quotients
$$ \big[E(\Phi,A),E(\Phi,B)\big]/E(\Phi,R,AB) $$ 
\noindent
described in Theorem 2,
without any stability conditions. This is a generalisation of the first 
half of \cite{RNZ4}, Lemma 15, to all Chevalley groups. 
Indeed, in view of Theorems 1 and 2 as a normal subgroup of 
$E(\Phi,R)$ the group $[E(\Phi,A),E(\Phi,B)]$ is generated by a
similar commutator for a rank 2 subsystem. This can be restated 
as follows.

\begin{The}
Let $R$ be any commutative ring with $1$, and let $A$ 
and $B$ be two sided ideals of $R$. Further, assume that $\Delta\le\Phi$
is a root subsystem containing $\A_2$ on long roots or $\C_2$. Then
the stability map
$$ \big[E(\Delta,A),E(\Delta,B)\big]/E(\Delta,R,AB) \map
\big[E(\Phi,A),E(\Phi,B)\big]/E(\Phi,R,AB) $$
\noindent
is surjective.
\end{The}

According to Theorem 4 modulo $E(\Phi,R,AB)$ the elementary
commutators $y_{\a}(a,b)$ behave as symbols. Theorems 3 and 5 
list some relations satisfied by these symbols. However, looking at 
the examples for which $[E(\Phi,A),E(\Phi,B)]$ was explicitly 
calculated, such as Dedekind rings of arithmetic type, \cite{MAS3,
MAS1, NV19}, it is easy to see that there must be further relations.

\begin{Prob} Give a presentation of
$\big[E(\Phi,A),E(\Phi,B)\big]/E(\Phi,R,AB)$
by generators and relations.
\end{Prob}

In the present paper we have generalised the main results 
of \cite{NZ1} to all Chevalley groups. It is natural to ask,
whether the same can be done also for the results of \cite{NZ3, NZ6}.
For the results of \cite{NZ3} this does not have much sense, since 
for {\it commutative\/} rings they already follow from the birelative
standard  commutator formula, and are already contained in 
\cite{RNZ2, Stepanov_universal, RNZ4}. The fact that they can be
proven by elementary calculations alone, without any use of 
localisation methods, is amusing, but does not have any tangible
implications.
\par
However, the analogues of results of \cite{NZ6} would be markedly 
new, and would have vital consequences. It is not even totally clear,
whether the triple congruences for subgroups of $\GL(n,R)$, such 
as established in \cite{NZ6}, Theorem 1, hold in this form in more 
general contexts, or should be replaced by fancier and longer ones.

\begin{Prob}
Prove analogues of \cite{NZ6}, Theorem\/ $1$, for Chevalley groups.
\end{Prob}

The partially relativised group $E(\Phi,B,A)={E(\Phi,A)}^{E(\Phi,B)}$
is the smallest $E(\Phi,B)$-normalised subgroup containing 
$E(\Phi,A)$. It is easy to derive from Theorem 1 that $E(\Phi,B,A)$ 
is generated by the elementary conjugates
$z_{\alpha}(a,b)={}^{x_{-\alpha}(b)}x_{\alpha}(a)$,
where $\alpha\in\Phi$, $a\in A$, $b\in B$. It is natural to ask,
whether this result can be improved further. Namely, can one 
limit the roots $\a$ here to roots in the special part of some
parabolic set of roots, as was done for $E(\Phi,R,A)$ by
van der Kallen and Stepanov, see \cite{vdK-group, Stepanov_calculus,
Stepanov_nonabelian}. 

\begin{Prob}
Prove that $E(\Phi,B,A)$ is generated by $E(\Phi,R)$ together with
the elementary conjugates $z_{\alpha}(a,b)={}^{x_{-\alpha}(b)}x_{\alpha}(a)$,
where  $a\in A$, $b\in B$, while $\a$ runs over the special part
of a fixed parabolic set of roots in $\Phi$. 
\end{Prob}

We are very grateful to 
Roozbeh Hazrat and Alexei Stepanov for 
extremely useful discussions at various stages of this work. Also, we
very much appreciate the help by Alexander Luzgarev who has sent 
us his neat {\tt Mathematica} package {\tt g2.nb}. Among other 
things, that package allowed us to generate explicit matrix form 
of root unipotents and Chevalley commutator formulae for the 
Chevalley group of type $\G_2$ in the adjoint representation, which
was crucial in getting the initial proof of Theorems 4 and 5 in this
case. Finally, we 
thank Anastasia Stavrova for her very pertinent questions during 
our seminar talk, and insistence.



\end{document}